\documentclass[a4paper,10pt]{article}

\usepackage{amsfonts,enumerate,array,amsmath,xcolor,amssymb,latexsym, amsthm, mathrsfs}
\usepackage{mathtools}
\usepackage{booktabs} 
\usepackage{caption} 
\usepackage{subcaption} 
\usepackage{graphicx}
\usepackage[all]{nowidow}
\usepackage[utf8]{inputenc}
\usepackage{tikz}
\usetikzlibrary{er,positioning,bayesnet}
\usepackage{multicol}
\usepackage{algpseudocode,algorithm,algorithmicx}
\usepackage{minted}
\usepackage{hyperref}
\usepackage{bbm}
\usepackage[inline]{enumitem} 
\usepackage{thmtools}
\usepackage{verbatim}
\usepackage[margin=3.0cm]{geometry}
\usepackage{orcidlink}

\numberwithin{equation}{section}

\newtheorem{theorem}{Theorem}[section]
\newtheorem{proposition}[theorem]{Proposition}
\newtheorem{lemma}[theorem]{Lemma}

\newtheorem{rem}[theorem]{Remark}
\theoremstyle{definition}

\newcommand{\E}{\mathbb{E}}
\newcommand{\N}{\mathbb{N}}
\newcommand{\I}{\mathbbm{1}}
\newcommand{\Pro}{\mathbb{P}}
\renewcommand{\P}{\mathbb{P}}
\newcommand{\R}{\mathbb{R}}

\newcommand{\teoref}[1]{Theorem~\ref{#1}}

\renewcommand{\eqref}[1]{(\ref{#1})}

\usepackage[
backend=biber,
style=numeric,
sorting=nyvt
]{biblatex}
\nocite{*}
\providecommand{\keywords}[1]
{
  \small	
  \noindent\textbf{Keywords.} #1
}

\begin{document}
	
\title{Weak approximation for Gaussian processes from renewal processes}
\author{Xavier Bardina\orcidlink{0000-0003-1803-3401}$^+$\footnote{Corresponding author.}, Salim Boukfal\orcidlink{0009-0007-5296-7087}$^+$, Marc Cano\orcidlink{0009-0005-6568-3513}$^+$ and Carles Rovira\orcidlink{0000-0001-9021-9804}$^\times$
\\{\footnotesize $^+$ Dept. Matem\`atiques, Edifici C, Universitat Aut\`onoma de Barcelona, 08193-Bellaterra}
\\{\footnotesize $^\times$ Facultat de Matem\`atiques, Universitat de Barcelona, Gran Via 585, 08007-Barcelona}
\\{\footnotesize E-mail: Xavier.Bardina@uab.cat, Salim.Boukfal@uab.cat, Marc.Cano@uab.cat, Carles.Rovira@ub.edu}
\footnote{All authors are supported by the grant PID2021-123733NB-I00 funded by MCIN/AEI/10.13039/501100011033. Salim Boukfal is also supported by the program of predoctoral grants AGAUR-FI (2025 FI-1 00119) Joan Oró from the Department of Research and Universities of the Government of Catalonia and the co-funding of the European Social Fund Plus (ESF+). Marc Cano is also supported by the predoctoral grant PRE2022-103868 funded by MCIN/AEI/ 10.13039/501100011033.}}
\date{}
\maketitle

\begin{abstract}
In \cite{bardina2023strong} a family of processes that converge strongly towards Brownian motion, defined from renewal processes, are constructed. In this paper we prove that some of these processes can be utilized to build approximations of Gaussian processes such as fractional Brownian motion or multiple Stratonovich integrals and we provide sufficient conditions on renewal processes to ensure that the convergence holds.  An illustrative example of such a Gaussian process is the fractional Brownian motion with any Hurst parameter. 
\end{abstract}

\section{Introduction}
Let $\{N(t),\,t\geq0\}$ a standard Poisson process and, for all $n\in\mathbb{N}$, consider the process:
\[h_n=\left\{h_n(t):=\frac{1}{\sqrt{n}}\int_0^{nt}(-1)^{N(u)}du,\,t\in[0,T]\right\}.\]

In 1982, Stroock (see \cite{stroock}) established that the processes $h_n$ converge in law to a standard Brownian motion. 
These processes were initially introduced by Kac in 1956 (see \cite{kac1974stochastic}) to derive a solution to the telegraph equation from a Poisson process.

Several extensions of the Kac–Stroock result have been investigated in the literature. These include modifying the processes $h_n$
to approximate Gaussian processes other than Brownian motion, establishing stronger forms of convergence, and relaxing the structural assumptions on the approximating processes in order to obtain convergence to Brownian motion in more general settings.

An extension in the first direction, that is, modifying the processes $h_n$ to obtain approximations of other Gaussian processes, is presented in \cite{delgado2000weak} by Delgado and Jolis, who consider a class of Gaussian processes encompassing the fractional Brownian motion. For a stochastic process $Y$ with representation:
\[Y=\{Y_t=\int_0^1K(t,r)dW_r,\,t\in[0,1]\},\]
where $W$ is a standard Brownian motion and the kernel $K:[0,1]\times[0,1]\longrightarrow\mathbb{R}$ satisfies certain properties, they define the process
\[{Y}_t^n={\sqrt{n}}\int_0^1K(t,r)(-1)^{N(nr)}dr.\]
Their result proves that ${Y}_t^n$ converges in law to $Y$ in the space $\mathcal{C}([0,1])$. 
Also in this direction, in \cite{bardina2000weak} we find results of convergence to Stratonovich multiple integrals.

The second direction of extension concerns proving convergence in a stronger sense than convergence in distribution on the space of continuous functions. Several works obtain such stronger convergence results, in which the realizations of the processes converge almost surely and uniformly on the unit interval to Brownian motion. This is achieved by modifying the processes so that they do not start with a strictly increasing behaviour. The resulting modified processes, commonly referred to as uniform transport processes, are
\[X_n(t)=\frac{1}{\sqrt{n}}(-1)^{A}\int_0^{nt}(-1)^{N(u)}du,\]
where $A\sim \textrm{Bern}\left(\frac{1}{2}\right)$ is independent of the Poisson process $N$.
Griego, Heath and Ruiz-Moncayo \cite{griego1971almost} showed that these processes
converge strongly and uniformly on bounded time intervals to
Brownian motion. 

The third direction is weakening the conditions of the approximating processes to find generalizations of the processes $(-1)^{N(u)}$ that also converge to Brownian motion.
In  \cite{bardina2016} we prove that if  instead of using the
Poisson process we consider a process with independent increments
which characteristic function satisfies some properties we obtain
also the convergence towards the law of a Brownian
motion. Between the examples of processes that satisfy these conditions
there are included the L\'evy processes.

Finally, in the work of Bardina and Rovira \cite{bardina2023strong}, the authors introduce a family of processes constructed from a renewal process, which converge almost surely to Brownian motion. The objective of the present paper is to study two results concerning the weak approximation of Gaussian processes arising from these realizations of Brownian motion. In this sense, our contribution extends the framework in two directions: firstly, by approximating Gaussian processes other than Brownian motion, and secondly, by establishing convergence results for more general processes than those originally considered by Kac and Stroock.

To achieve this, we follow an approach analogous to that of Delgado and Jolis \cite{delgado2000weak} for our first main result (\teoref{teo1}), while the proof of our second result (\teoref{teo2}) is inspired by the ideas developed in Bardina and Jolis \cite{bardina2000weak}. To the best of the authors’ knowledge, these are the first weak convergence results obtained from processes constructed using renewal processes available in the literature. They shed light on the conditions required to ensure convergence once the convenient properties of the exponential inter-arrival distribution of the Poisson process are no longer present.

The paper is organized as follows. In Section 2, we introduce the renewal processes and the kernels that will be used to construct our approximations. Section 3 is devoted to the convergence towards integrals with respect to a deterministic kernel, while Section 4 addresses the convergence towards multiple integrals. Finally, an appendix is included, where we discuss renewal processes with geometrically distributed inter-arrival times.

\section{Renewal process}

Consider $\left\lbrace U_m\right\rbrace_{m\geq 1}$ a sequence of independent random variables which take only nonnegative values. Furthermore, assume that they are identically distributed with $\Pro\left(U_1=0\right)<1$ and $\E\left[\left(U_1\right)^l\right]<\infty$, for $l$ large enough (for our purposes, $l \geq 4$ suffices).

For each $k\geq 1$, consider the renewal sequence $S_k = U_1+\cdots +U_k$ and the counting renewal function 
\[
L(t) = \sum_{k=1}^{\infty} \mathbbm{1}_{[0,t]}(S_k).
\]
This renewal function counts the number of renewals in $[0,t]$.

\vspace{11pt}
Set $\left\lbrace \eta_m \right\rbrace_{m\geq 0}$ a sequence of independent identically distributed random variables following a $Bernoulli(\frac{1}{2})$ distribution and independent of the sequence $\left\lbrace U_m\right\rbrace_{m\geq 1}$. Then, we can construct the following renewal reward process
\[
T(t) = \eta_0 + \sum_{k=1}^{\infty}\eta_k\mathbbm{1}_{[0,t]}(S_k) = \sum_{l=0}^{L(t)}\eta_l.
\]

Given a strictly positive function $\beta$, we can define
\begin{equation*}
T_n(t) = T_{\beta(n)}(t) = T\left(\frac{t}{\beta(n)}\right) = \eta_0 + \sum_{k=1}^{\infty}\eta_k\mathbbm{1}_{\left[0,\frac{t}{\beta(n)}\right]}(S_k). 
\end{equation*}

Notice that, setting $U_m^n=\beta(n)\times U_m$, for all $m\geq 1$, we have that
\[
\beta(n)\times S_k = U_1^n+\cdots +U_k^n.
\]

We will consider the kernels
\begin{equation}\label{thekernel}
\displaystyle\theta_n(x):= \frac{1}{G(n)}(-1)^{T\left(\frac{x}{\beta(n)}\right)}
,
\end{equation}
where
\begin{equation*}
G(n) = \left(\beta(n)\frac{\E\left[\left(U_1\right)^2\right]}{\E\left[U_1\right]}\right)^{\frac{1}{2}},
\end{equation*}
with 
\[
\sum_{n\geq 1}\beta(n)<\infty.
\]

Let us recall that, under these hypotheses, in \cite{bardina2023strong} we prove that the processes 
\begin{equation}\label{bmcon}
    x_n(t) = \int_0^t \theta_n(u) du
\end{equation}
converge almost surely to the Brownian motion.

If, in addition, we assume that $U_1$ has a failure rate function bounded away from 0,
we can go a step further and show the convergence of processes constructed from the $x_n$ such as stochastic integrals or iterated integrals with respect to these (see Theorems \ref{teo1} and \ref{teo2}). 
More precisely, we will assume that the cumulative distribution function $F$ of $U_1$ satisfies the following hypothesis:

\begin{enumerate}
    \item[(G)] $F$ is given by
            $$F(t) = 1- \exp\left\{-\int_0^t r(s)ds\right\}, \quad t \geq 0, $$ 
            (and $F(t) = 0$ otherwise) for some non-negative function $r$ such that $\int_0^\infty r(s)ds = \infty$ and for which there is a positive constant $\lambda > 0 $ such that $r(s) \geq \lambda$ for every $s\geq 0$.
\end{enumerate}

\section{Convergence to Gaussian processes represented by a stochastic integral of a deterministic kernel }

Let $X=\left\lbrace X_t,~t\in[0,1] \right\rbrace$ be a centered Gaussian process with covariance function given by 
\begin{equation*}
    Cov\left(X_t,X_s\right) = \int_0^1 K(t,x)K(s,x)dx,
\end{equation*}
where $K\colon [0,1]\times [0,1]\to \R$ is a function satisfying the following assumptions:
\begin{equation}\label{assumptions}
(H)\left\lbrace\begin{matrix*}[l]
    (i)& & K~\text{is measurable and } K(0,x) = 0~\forall x\in [0,1].\\ \\
    (ii)& & \text{There exists a continuous increasing function }G\colon [0,1]\to\R \text{ and} \\ &&\text{a constant }\alpha>0~\text{such that for any }0\leq s<t\leq 1,\\ \\ 
    & &\displaystyle\int_0^1 \left(K(t,x)-K(s,x)\right)^2dx\leq \left(G(t)-G(s)\right)^{\alpha}.
\end{matrix*}\right.
\end{equation}

In our first theorem, we establish that the law of a continuous Gaussian process, represented by a stochastic integral of a deterministic kernel with respect to a standard Wiener process, can be weakly approximated by the law of processes $\left\lbrace Y^n\right\rbrace$ constructed using the kernels $\theta_n$.
More precisely, our first theorem states as follows:

\begin{theorem}\label{teo1}
     Suppose that $K$ satisfies hypothesis $(H)$ given in (\ref{assumptions}) and that 
     the distribution function of $U_1$ satisfies hypothesis $(G)$. Let $\left\lbrace Q^n \right\rbrace$ be the family of laws in $\mathcal{C}([0,1])$ of the processes $\left\lbrace Y^n\right\rbrace$ given by 
     \begin{equation*}
         Y_t^n = \int_0^1 K(t,x)\theta_n(x)dx,~~\text{for any }t\in[0,1],
     \end{equation*}
     where $\theta_n$ are the kernels defined in (\ref{thekernel}). Then, $\left\lbrace Q^n \right\rbrace$ converges weakly in $\mathcal{C}([0,1])$, as $n\to\infty$, to a law $Q$, that has the same finite-dimensional distributions as the Gaussian process $X$. 
\end{theorem}

\begin{rem} This family of processes includes the fractional Brownian motion (fBm) because if we consider the kernel
$K(t,s)=K^H(t,r)\mathbbm{1}_{[0,t]}(r)$ where {\small
$$K^H(t,r)=C_H(t-r)^{H-\frac12}+C_H\left(\frac12-H\right)\int_r^t(u-r)^{H-\frac32}\left(1-\left(\frac{r}{u}\right)^{\frac12-H}\right)du,$$
}for a certain constant $C_H$, then the fBm admits the
representation
$$B^H=\{B_t^H=\int_0^1K(t,r)dW_r,\,t\in[0,1]\},$$
and it is easy to see that the kernel $K$ satisfies the two previous
hypothesis.
\end{rem}

To prove \teoref{teo1}, we will show that the sequence $Y^n$ is tight and that the finite dimensional distributions converge to those of $X$. In order to check both conditions, as discussed in \cite{delgado2000weak}, it is sufficient to show that
\begin{equation}\label{moment condition}
    \sup_n \E\left[\left|\int_0^1f(x)\theta_n(x)dx\right|^m\right]\leq C_m\left(\int_0^1f^2(x)dx\right)^{\frac{m}{2}},
\end{equation}
for some constants $m > 2$, $C_m > 0$ and for every $f \in L^2([0,1])$ and that the finite dimensional distributions of the sequence $x_n$ defined in \eqref{bmcon} converge to those of the Brownian motion. The latter is a consequence of the fact that the sequence $x_n$ converges to the Brownian motion in the strong sense (in $\mathcal{C}([0,1])$), which in turn implies weak convergence and thus, weak convergence of the finite dimensional distributions. Hence, the only thing left to do is to show \eqref{moment condition}. A first step towards that direction is given by Lemma \ref{primers pas} below. 

\begin{lemma}\label{primers pas}
    For any even $m \in \mathbb{N}$ and $f \in L^2([0,1])$, we have that
    \begin{align}\label{fita moments renewal}
        &\E\left[ \left( \int_0^1 f(x)\theta_n(x)dx \right)^m \right] \nonumber \\
        &= \frac{m!}{G^m(n)} \int_{[0,1]^m} \left(\prod_{j=1}^m f(x_j) \right) \Pro\left\{ L\left(\frac{x_{2j}}{\beta(n)} \right) -  L\left(\frac{x_{2j-1}}{\beta(n)} \right) = 0 \colon j=1,...,\frac{m}{2} \right\} \I_{\{x_1 \leq ... \leq x_m\}}dx_1 ... dx_m.
    \end{align}
\end{lemma}
\begin{proof}
    We first start by observing that
    \begin{align*}
        &\E\left[ \left( \int_0^1 f(x)\theta_n(x)dx \right)^m \right] =\frac{1}{G^m(n)} \int_{[0,1]^m} \left( \prod_{j=1}^m f(x_j) \right) \E\left[ (-1)^{\sum_{j=1}^m T_n(x_j)} \right] dx_1 ... dx_m \\
        &=\frac{m!}{G^m(n)} \int_{[0,1]^m} \left( \prod_{j=1}^m f(x_j) \right) \E\left[ (-1)^{\sum_{j=1}^m T_n(x_j)} \right] \I_{\{x_1 \leq ... \leq x_m\}} dx_1 ... dx_m \\
        &=\frac{m!}{G^m(n)} \int_{[0,1]^m} \left( \prod_{j=1}^m f(x_j) \right) \E\left[ (-1)^{\sum_{j=1}^{m/2} \left(T_n(x_{2j}) - T_n(x_{2j-1})\right)} \right] \I_{\{x_1 \leq ... \leq x_m\}} dx_1 ... dx_m,
    \end{align*}
    where we have used that $T$ is integer-valued.\\
    For $0 \leq x_1 \leq ... \leq x_m$, we can write
    \begin{equation*}
        \sum_{j=1}^{m/2} \left(T_n(x_{2j}) - T_n(x_{2j-1}) \right) = \sum_{k=1}^\infty \eta_k \I_{\bigcup_{j=1}^{m/2}\left(\frac{x_{2j-1}}{\beta(n)}, \frac{x_{2j}}{\beta(n)} \right]}(S_k),
    \end{equation*}
    which (almost surely) contains a finite number of summands since $S_k \uparrow \infty$ almost surely. Hence, by letting
    \begin{equation*}
        A = \left\{ \sum_{k=1}^\infty \I_{\bigcup_{j=1}^{m/2}\left(\frac{x_{2j-1}}{\beta(n)}, \frac{x_{2j}}{\beta(n)} \right]}(S_k) > 0  \right\},
    \end{equation*}
    we see that the events
    \begin{equation}\label{conjunts quasi iguals}
        A \quad \text{and} \quad A \cap \left\{ \sum_{k=1}^\infty \I_{\bigcup_{j=1}^{m/2}\left(\frac{x_{2j-1}}{\beta(n)}, \frac{x_{2j}}{\beta(n)} \right]}(S_k) <  \infty  \right\}
    \end{equation}
    are equal up to a null set. Now let $\Pro_A(\cdot) = \Pro(\cdot | A)$, we then have that
    \begin{align*}
        &\E\left[ (-1)^{\sum_{j=1}^{m/2} \left(T_n(x_{2j}) - T_n(x_{2j-1})\right)} \Big| A \right]\\
        &= \Pro_A\left\{ \sum_{j=1}^{m/2} \left(T_n(x_{2j}) - T_n(x_{2j-1})\right) \text{ is even} \right\} - \Pro_A\left\{ \sum_{j=1}^{m/2} \left(T_n(x_{2j}) - T_n(x_{2j-1})\right) \text{ is odd} \right\}.
    \end{align*}
    On one hand,
    \begin{align*}
        &\Pro_A\left\{ \sum_{j=1}^{m/2} \left(T_n(x_{2j}) - T_n(x_{2j-1})\right) \text{ is even} \right\}\\
        &= \sum_{i=1}^\infty \Pro_A\left\{ \sum_{j=1}^{m/2} \left(T_n(x_{2j}) - T_n(x_{2j-1})\right) \text{ is even}, \sum_{k=1}^\infty \I_{\bigcup_{j=1}^{m/2}\left(\frac{x_{2j-1}}{\beta(n)}, \frac{x_{2j}}{\beta(n)} \right]}(S_k) = i \right\} \\
        &= \sum_{i=1}^\infty \sum_{k_1 < ... < k_i} \Pro_A\left\{ \eta_{k_1}+...+\eta_{k_i} \text{ is even}, S_{k_1},...,S_{k_i} \in \bigcup_{j=1}^{m/2}\left(\frac{x_{2j-1}}{\beta(n)}, \frac{x_{2j}}{\beta(n)} \right] \right\} \\
        &= \sum_{i=1}^\infty \sum_{k_1 < ... < k_i} \Pro_A\left\{ \eta_{k_1}+...+\eta_{k_i} \text{ is even} \right\}\Pro_A\left\{ S_{k_1},...,S_{k_i} \in \bigcup_{j=1}^{m/2}\left(\frac{x_{2j-1}}{\beta(n)}, \frac{x_{2j}}{\beta(n)} \right] \right\} \\
        &= \frac{1}{2}\sum_{i=1}^\infty \sum_{k_1 < ... < k_i} \Pro_A\left\{ S_{k_1},...,S_{k_i} \in \bigcup_{j=1}^{m/2}\left(\frac{x_{2j-1}}{\beta(n)}, \frac{x_{2j}}{\beta(n)} \right] \right\} \\
        &= \frac{1}{2}\sum_{i=1}^\infty  \Pro_A\left\{ \sum_{k=1}^\infty \I_{\bigcup_{j=1}^{m/2}\left(\frac{x_{2j-1}}{\beta(n)}, \frac{x_{2j}}{\beta(n)} \right]}(S_k) = i \right\} \\
        &= \frac{1}{2}.
    \end{align*}
    In the second equality the fact that the sets in \eqref{conjunts quasi iguals} are the same up to a null set is being used (so that only a finite number of $S_{k_j}$ and hence, of $\eta_{k_j}$, appear), in the third one we have used that the sequences $\{\eta_k\}_k$ and $\{U_k\}_k$ are independent, in the fourth one the fact that the probability of a binomial random variable with probability of success $1/2$ being even is $1/2$ (no matter how many trials are performed) is being used and in the last step we use that  
    \begin{equation*}
        \sum_{i=1}^\infty \Pro_A\left\{ \sum_{k=1}^\infty \I_{\bigcup_{j=1}^{m/2}\left(\frac{x_{2j-1}}{\beta(n)}, \frac{x_{2j}}{\beta(n)} \right]}(S_k) = i \right\} = 1.
    \end{equation*}
    Similarly, we see that
    \begin{equation*}
        \Pro_A\left\{ \sum_{j=1}^{m/2} \left(T_n(x_{2j}) - T_n(x_{2j-1})\right) \text{ is odd} \right\} = \frac{1}{2},
    \end{equation*}
    so that $\E\left[ (-1)^{\sum_{j=1}^{m/2} \left(T_n(x_{2j}) - T_n(x_{2j-1})\right)} \Big| A \right] = 0$. Finally, using that the renewal process $L$ is non-decreasing,
    \begin{align*}
        \E\left[ (-1)^{\sum_{j=1}^{m/2} \left(T_n(x_{2j}) - T_n(x_{2j-1})\right)} \right] 
        &= \E\left[ (-1)^{\sum_{j=1}^{m/2} \left(T_n(x_{2j}) - T_n(x_{2j-1})\right)} \Big| A^c \right] \Pro(A^c) \\
        &= \Pro\left\{ \sum_{k=1}^\infty \I_{\bigcup_{j=1}^{m/2}\left(\frac{x_{2j-1}}{\beta(n)}, \frac{x_{2j}}{\beta(n)} \right]}(S_k) = 0 \right\} \\
        &= \Pro\left\{ \sum_{j=1}^{m/2} \left(L\left( \frac{x_{2j}}{\beta(n)} \right) - L\left( \frac{x_{2j-1}}{\beta(n)} \right) \right) =0 \right\} \\
        &= \Pro\left\{ L\left(\frac{x_{2j}}{\beta(n)} \right) -  L\left(\frac{x_{2j-1}}{\beta(n)} \right) = 0 \colon j=1,...,\frac{m}{2} \right\},
    \end{align*}
    which finishes the proof of the lemma.
\end{proof}

Without any further assumptions on the renewal process $L$ or on the sequence of interarrival times $\{U_k\}_k$ we cannot proceed any further. When the random variables $\{U_k\}_k$ are exponentially distributed, $L$ becomes a Poisson process and one can proceed as in \cite{delgado2000weak} to obtain \eqref{moment condition}. The same can be said when the interarrival times are geometric random variables of parameter $p \in (0,1)$. In this case, it can be shown (see \cite{Samimi2013ANA}\footnote{For the sake of completeness and clarity, and due to the difficulty of accessing the original sources, we sketch the proofs in the Appendix.}) that $L$ has independent increments which are distributed as
\begin{equation*}
    \Pro\{  L(t) - L(s) = j \} = \binom{[t] - [s]}{j}p^j (1-p)^{[t]-[s]-j}, \quad j \in \{0,1,..,[t]-[s]\},
\end{equation*}
for $0 \leq s \leq t$ (being $[\cdot]$ the integer part function). In particular,
\begin{equation*}
    \Pro\{  L(t) - L(s) = 0 \} = (1-p)^{[t]-[s]} \leq \frac{1}{1-p} e^{-a(t-s)}, \quad a = -\log(1-p),
\end{equation*}
so that similar computations as the ones shown in \cite{delgado2000weak} lead to bounds like \eqref{moment condition}.

In general, we cannot expect $L$ to have independent increments nor that the probability of them being 0 to decay exponentially. In the following lemma, we see that the requirement of $U_1$ having a failure rate function bounded away from 0 is sufficient in order to ensure that the probability in \eqref{fita moments renewal} decays exponentially, allowing us to obtain the bounds in \eqref{moment condition}.

\begin{lemma}\label{lema fita per poisson}
    Assume that $U_1$ has cumulative distribution function $F$  satisfying hypothesis $(G)$. 
   Then \eqref{moment condition} holds for any even $m \in \mathbb{N}$.
\end{lemma}
\begin{proof}
    By Theorem 2 in \cite{miller}, one can construct a Poisson process $N$ with intensity $\lambda$ whose arrival times are obtained by thinning $\{S_0,S_1,S_2,...\}$. More precisely, there is a Poisson process of intensity $\lambda$, $N$, such that
    \begin{equation*}
        \{T_0, T_1, ...\} \subset \{S_0, S_1, ...\}, \quad T_i = \inf\{t > 0 \colon N(t) \geq i \}, 
    \end{equation*}
    almost surely. This in particular implies that $L$ is bounded below by $N$ with probability 1 and thus, that if $L$ has not performed any jumps on $(s,t]$, $0 \leq s \leq t$, then $N$ has not performed any jumps on $(s,t]$ as well. Bearing this in mind, the probability in \eqref{fita moments renewal} can be bounded above by
    \begin{align*}
        \Pro\left\{ N\left(\frac{x_{2j}}{\beta(n)} \right) -  N\left(\frac{x_{2j-1}}{\beta(n)} \right) = 0 \colon j=1,...,\frac{m}{2} \right\} &= \prod_{j=1}^{m/2} \Pro\left\{ N\left(\frac{x_{2j}}{\beta(n)} \right) -  N\left(\frac{x_{2j-1}}{\beta(n)} \right) = 0 \right\} \\
        &= \prod_{j=1}^{m/2} \exp\left\{ -\frac{x_{2j} - x_{2j-1}}{\beta(n)}  \right\},
    \end{align*}
    where in the last line we have used that the Poisson process has independent increments. From here, the same computations as in \cite{delgado2000weak} lead to the desired results.
\end{proof}

\begin{rem}
    If, for instance, $U_1$ has the same law as the absolute value of a standard Gaussian random variable, then its density and cumulative distribution functions are given by, respectively, $f(t) = 2\phi(t)$ and $F(t) = 2\Phi(t) - 1$ for $t \geq 0$ and where $\phi$ and $\Phi$ are the density and the cumulative distribution functions of a standard Gaussian variable, respectively. It follows from  \cite[Formula 7.1.13, p.298]{abramowitz1968handbook} that
    \begin{equation*}
        r(t) \geq t + \sqrt{t^2 + \frac{8}{\pi}} \geq 2\sqrt{\frac{2}{\pi}} > 0.
    \end{equation*}
    Implying that such $U_1$ verifies the hypotheses of Lemma \ref{lema fita per poisson}.

    The same can be said if $U_1$ is uniformly distributed over $[0,1]$ since, in this case,
    \begin{equation*}
        r(t) = \frac{1}{1-t}, \quad t \in [0,1)
    \end{equation*}
    and $r(t) = +\infty$ for $t \geq 1$. In this case, we have $r(t) \geq 1$ for all $t \geq 0.$

    Another example can be given by assuming that $U_1$ is distributed as a gamma random variable of parameters $\alpha \in (0,1)$ and $\lambda > 0$. More precisely, if $U_1$ is absolutely continuous with density function
    \begin{equation*}
        f(x) = \frac{\lambda^{\alpha}}{\Gamma(\alpha)}x^{\alpha - 1} e^{-\lambda x}\I_{(0,\infty)}(x),
    \end{equation*}
    where $\Gamma$ denotes the Euler gamma function. In this case we have that
    \begin{equation*}
        r(t) = \frac{f(t)}{S(t)}, \quad S(t) = \int_t^\infty f(x)dx, \quad  t > 0,
    \end{equation*}
    and $r(0) = +\infty$. It is easy to show that
    \begin{equation*}
        \lim_{t \to \infty} r(t) = \lambda > 0,
    \end{equation*}
    hence, if we show that $r(t)$ is strictly decreasing in $(0,\infty)$, we will have that $r(t) > \lambda $ for each $t \in [0,\infty)$. A direct computation shows that, for $t > 0$,
    \begin{equation*}
        r'(t) = \frac{f(t)}{S^2(t)}\left[ \left(\frac{\alpha - 1}{t} - \lambda\right)S(t) + f(t) \right].
    \end{equation*}
    Thus, $r'(t) < 0$ if, and only if, 
    \begin{equation*}
        G(t) = \frac{t f(t)}{\lambda t - \alpha + 1} - S(t) < 0
    \end{equation*}
    for all $t \geq 0$. But the latter is a consequence of the fact that $G(0) = -1 < 0$, that
    \begin{equation*}
        G'(t) = (1-\alpha)\frac{2f(t)}{(\lambda t - \alpha + 1)^2} > 0
    \end{equation*}
    and that $\lim_{t \to \infty} G(t) = 0$. 
\end{rem}

\section{Convergence to multiple integrals}

In this section, we show that our approximating processes satisfy the conditions established in \cite{bardina2000weak} to ensure weak convergence towards multiple Stratonovich integrals of functions of the form $f(t_1,\ldots,t_n)= f_1(t_1)\cdots f_n(t_n)1_{\{t_1<t_2<\cdots<t_n\}}$, with $f_i\in L^2([0,T])$ for any $i=1,\ldots,n$.

The theorem states as follows:

\begin{theorem}\label{teo2}
    Let $f_i\in L^2([0,1])$ for all $i\in\left\lbrace 1,\ldots, l \right\rbrace$ and define
\[
\begin{matrix*}[l]
    \displaystyle Y_1^n(t) &=& \displaystyle\int_0^t f_1(x)\theta_n(x)dx,~~\text{and for }k\in \lbrace 2,\ldots, l \rbrace,\\ \\
    \displaystyle Y_k^n(t) &=& \displaystyle\int_0^t f_k(x)Y_{k-1}^n(x)\theta_n(x)dx,
\end{matrix*}
\]
where $\theta_n$ are the kernels defined in (\ref{thekernel})  and 
the distribution function of $U_1$ satisfies hypothesis $(G)$. 
Then,
\[
\mathscr{L}\left(Y_1^n,\ldots,Y_l^n\right)\xlongrightarrow[]{w}\mathscr{L}\left(Y_1,\ldots, Y_l\right)
\]
in the space $\left(\mathcal{C}_0\left([0,1]\right)\right)^l$ when $n\to\infty$, where
\[
\begin{matrix*}[l]
    \displaystyle Y_1(t) &=& \displaystyle\int_0^t f_1(x)dW_x,~~\text{and for }k\in \lbrace 2,\ldots, l \rbrace,\\ \\
    \displaystyle Y_k(t) &=& \displaystyle\int_0^t f_k(x)Y_{k-1}(x)\circ dW_x,
\end{matrix*}
\]
where the last integral is in the Stratonovich sense.
\end{theorem}

\vspace{11pt}

We now prove \teoref{teo2}. Similarly to \teoref{teo1}, the proof of this theorem relies on proving an inequality similar to \eqref{moment condition}, which is the content of the following lemma.

\begin{lemma}\label{main2}
Under the same hypothesis as in Lemma \ref{lema fita per poisson}, for any non-negative $g\in L^2\left([0,1]\right)$ and any $m \in \mathbb{N}$, there exists a constant $K$, that only depends on $m$ and on the $L^2\left([0,1]\right)$-norm of the function $g$, such that, for all $0\leq s\leq t\leq 1$,
\begin{align*}
    &\int_{[s,t]^4\times [0,t]^{4(m-1)}}g(x_1)\cdots g(x_{4m})\left|\E\left[\theta_{n}(x_1)\cdots\theta_{n}(x_{4m})\right]\right|\mathbbm{1}_{\left\lbrace x_1\leq\cdots\leq x_{4m} \right\rbrace} dx_1\cdots dx_{4m}\\ &\leq K\left(\int_s^t g^2(x)dx\right)^2,
\end{align*}
where $\{\theta_{n}\}_n$ are our kernels.
\end{lemma}
\begin{proof}
    We want to show that
    \begin{align*}
        \frac{1}{G^{4m}(n)}&\int_{[s,t]^4\times [0,t]^{4(m-1)}}g(x_1)\cdots g(x_{4m})\left|\E\left[(-1)^{T_n\left(x_1\right)}\cdots(-1)^{T_n\left(x_{4m}\right)}\right]\right|\mathbbm{1}_{\left\lbrace x_1\leq\cdots\leq x_{4m} \right\rbrace}dx_1\cdots dx_{4m}\\
        &\leq  K\left(\int_s^t g^2(x)dx\right)^2.
    \end{align*}
    To show this we can proceed as in Lemma \ref{primers pas}, use Theorem 2 in \cite{miller} and follow the same computations as in \cite{delgado2000weak} again to obtain that left-hand side of the latter equation can be bounded by
    \begin{equation*}
        C_m \left(\int_s^t g^2(x)dx\right)^2\left(\int_0^t g^2(x)dx\right)^{2(m-1)} \leq C_m\left(\int_s^tg^2(x)dx\right)^2\|g\|_2^{4(m-1)},
    \end{equation*}
    for some positive constant $C_m$ which only depends on $m$. The result follows by setting $K = C_m ||g||_2^{4(m-1)}$.
\end{proof}

With this result proved, we can proceed as in the paper of Bardina and Jolis \cite{bardina2000weak} in order to achieve the proof of \teoref{teo2}. 

\section*{Appendix}

\subsection*{The renewal process with geometric interarrival times}

In this section we shall see some properties of the renewal process $L$ when the random variables $\{U_k\}_k$ are geometric with parameter $p \in (0,1)$. That is,
\begin{equation*}
    \Pro\{U_1 = k\} = pq^{k-1}, \quad k \in \mathbb{N},\quad q = 1-p.
\end{equation*}
Since $S_n$ takes values in $\{n, n+1, ...\}$ (in fact, it is well known that is distributed as a negative binomial random variable with parameters $n$ and $p$), one has that $L(t) = L([t])$ for any $t \geq 0$ and that $L(t) - L(s) \leq [t] - [s]$ for any $0 \leq s \leq t$ since the increment $L(t) - L(s)$ is at most the number of integer numbers in $(s,t]$. 

The following result determines completely the law of the increments of $L$.
\begin{lemma}
    For any $0 \leq s \leq t$ and any $m,n \in \mathbb{N}_0 = \mathbb{N}\cup\{0\}$ with $m \leq [s]$ and $n \leq [t] - [s]$, we have that
    \begin{equation*}
        \Pro\{L(s) = m, L(t)-L(s) = n\} = \binom{[s]}{m}\binom{[t]-[s]}{n} q^{[t]} \left( \frac{p}{q}\right)^{m+n}. 
    \end{equation*}
\end{lemma}
\begin{proof}
    One has to consider several cases depending on the values of $m$ and $n$. More precisely, the cases where $(m,n)=(0,0)$, $(m,n) = (0,1)$, $(m,n)=(1,0)$, $(m,n) = (0,n)$ with $n\geq 2$, $(m,n)=(m,0)$ with $m \geq 1$, $(m,n)=(m,1)$ with $m \geq 1$ and, lastly, the case $m \geq 1$ and $n \geq 2$ have to be considered separately. We will only detail the computations for the latter, that is, we will be assuming that $m$ and $n$ are such that
    \begin{equation*}
        1\leq m \leq [s], \quad 2 \leq n, \quad m+n \leq [t].
    \end{equation*}
    We have that
    \begin{equation*}
        \Pro\{L(s) = m, L(t)-L(s) = n\} = \Pro\{L(s) = m, L(t) = n + m\} = \Pro\{S_m \leq s < S_{m+1} \leq S_{m+n} \leq t < S_{m+n+1}\}.
    \end{equation*}
    Let us define the random variables,
    \begin{align*}
        X_1 &= S_m = Y_1,\\
        X_2 &= S_{m+1} = Y_1 + Y_2, \\
        X_3 &= S_{m+n} = Y_1 + Y_2 + Y_3, \\
        X_4 &= S_{m+n+1} = Y_1 + Y_2 + Y_3 + Y_4,
    \end{align*}
    where the random variables $Y_1,Y_2,Y_3,Y_4$ are independent with laws 
    \begin{equation*}
        Y_1 \sim \text{NB}(m,p), \quad Y_2 \sim \text{Geom}(p), \quad Y_3 \sim \text{NB}(n-1,p), \quad Y_4 \sim \text{Geom}(p),
    \end{equation*}
    (NB$(m,p)$ and Geom$(p)$ denote, respectively, a negative binomial random variable of parameters $m$ and $p$ and a geometric random variable of parameter $p$).

    The next step is to write the set 
    \begin{equation*}
        D = \left\{(x_1, x_2, x_3, x_4) \in \N^4 \colon m \leq x_1 \leq s < x_2 < x_3 \leq t < x_4 \right\}
    \end{equation*}
    as the set of points $(x_1, x_2, x_3, x_4) \in \N^4$ such that
    \begin{equation*}
        \quad m \leq x_1 \leq [s], \quad [s]+1 \leq x_2 \leq x_3-1, \quad [s]+2 \leq x_3 \leq [t],\quad [t]+1 \leq x_4.
    \end{equation*}
    With all this, we have that
    \begin{align*}
        &\P\{S_m \leq s < S_{m+1} \leq S_{m+n} \leq t < S_{m+n+1}\} = \P\{(X_1,X_2,X_3,X_4) \in D\} \\
        &= \sum_{(x_1,...,x_4) \in D} \P\{X_1 = x_1, X_2 = x_2, X_3 = x_3, X_4 = x_4\} \\
        &= \sum_{(x_1,...,x_4) \in D} \P\{Y_1 = x_1, Y_2 = x_2 - x_1, Y_3 = x_3 - x_2, Y_4 = x_4 - x_3 \} \\
        &= \sum_{(x_1,...,x_4) \in D} \binom{x_1 - 1}{m-1}p^m q^{x_1 - m} p q^{x_2 - x_1 - 1}\binom{x_3 - x_2 - 1}{n - 2}p^{n - 1}q^{x_3 - x_2 - n + 1} p q^{x_4 - x_3 - 1}\\
        &=\sum_{(x_1,...,x_4) \in D} \binom{x_1 - 1}{m-1}\binom{x_3 - x_2 - 1}{n - 2} p^{m+n+1} q^{x_4 - m - n -1},
    \end{align*}
     where  we use that $\binom{k}{m} = 0$ if $k <m$.
    This last sum can be factored as a product of three sums, one depending only on $x_1$, another depending on $x_2$ and $x_3$ and another one depending only on $x_4$. The latter is a geometric sum, so we will not show how to compute it. For the former (the sum indexed by $x_1$), we will use the fact that
    \begin{equation*}
        \sum_{k=0}^m \binom{n+k}{n} = \binom{m+n+1}{n+1}.
    \end{equation*}
    Bearing this in mind, we find that
    \begin{equation*}
        \sum_{x_1 = m}^{[s]} \binom{x_1 - 1}{m-1} = \sum_{x_1 = 0}^{[s]-m}\binom{m - 1 +x_1}{m-1} = \binom{[s]}{m}.
    \end{equation*}
    Thus, so far we got that
    \begin{equation*}
        \P\{L(s) = m, L(t) - L(s) = n\} = p^{m+n}q^{[t]-m-n}\binom{[s]}{m}\sum_{x_3 = [s]+2}^{[t]} \sum_{x_2 = [s]+1}^{x_3 - 1} \binom{x_3 - x_2 - 1}{n-2}.
    \end{equation*}
    This last sum can be treated by similar means as the sum indexed by $x_1$, but now using that
    \begin{equation*}
        \sum_{j=0}^n \binom{j}{m} = \binom{n+1}{m+1},
    \end{equation*}
   and that $n \geq 2$,
    \begin{align*}
        \sum_{x_3 = [s]+2}^{[t]} \sum_{x_2 = [s]+1}^{x_3 - 1} \binom{x_3 - x_2 - 1}{n-2} &= \sum_{x_2 = [s]+1}^{[t]-1} \sum_{x_3 = x_2+1}^{[t]} \binom{x_3 - x_2 - 1}{n-2} \\
        &= \sum_{x_2 = [s]+1}^{[t]-1} \sum_{x_3 = 0}^{[t]-x_2-1}\binom{x_3}{n-2} \\
        &= \sum_{x_2 = [s]+1}^{[t]-1} \binom{[t] - x_2}{n-1}\\
        &= \sum_{x_2 = 1}^{[t]-[s]-1} \binom{x_2}{n-1} \\
        &= \sum_{x_2 = 0}^{[t]-[s]-1} \binom{x_2}{n-1}\\
        &= \binom{[t]-[s]}{n}.
    \end{align*}
\end{proof}
This result, in particular, shows that (by setting $s = 0$ and $m = 0$),
\begin{align*}
    \P\{L(t) = n\} = q^{[t]}\binom{[t]}{n} \left(\frac{p}{q} \right)^n.
\end{align*}
Moreover, one can easily see as well that
\begin{align*}
    \P\{L(t) - L(s) = n\} &= \sum_{m=0}^{[s]} \P\{L(s) = m, L(t)-L(s)=n\} \\
    &=\sum_{m=0}^{[s]}\binom{[s]}{m}\binom{[t]-[s]}{n} q^{[t]} \left( \frac{p}{q}\right)^{m+n}\\
    &= \binom{[t]-[s]}{n}q^{[t]}\left(\frac{p}{q} \right)^{n} \left(1 + \frac{p}{q} \right)^{[s]} \\
    &= \binom{[t] - [s]}{n}q^{[t]-[s]-n}p^n.
\end{align*}
Which shows that, for any $0 \leq s \leq t$, $L(s)$ and $L(t) - L(s)$ are independent and that, when $0 \leq s \leq t$ with $s,t \in \N_0$, the law of $L(t) - L(s)$ is the same as the law of $L(t - s)$ (observe that this need not be true when $s$ and $t$ are not in $\N_0$).

Finally, one can show that the increments of the process $L$ are independent.
\begin{proposition}
    The increments of $L$ are independent.
\end{proposition}
\begin{proof}[Sketch of the proof]
    Since $L(t) = L([t])$ for all $t \geq 0$, we only need to show that the increments are independent when the process is considered in $\N_0$. That is, we only need to show that, for every $k \in \N$ and any $0 \leq t_0 \leq t_2 \leq ... \leq t_k$ with $t_j \in \N_0$ for every $j \in \{0,...,k\}$, the random variables $L(t_0), L(t_1) - L(t_0),...,L(t_k) - L(t_{k-1})$ are independent. The proof of this last fact is quite similar to the one for the case in which the random variables $\{U_j\}_j$ are exponentially distributed (that is, when $L$ is a Poisson process).

    Let us define
    \begin{equation*}
        U_1^t = S_{L(t) + 1}-t, \quad U_k^t = U_{L(t) + k}, \quad k \geq 2.
    \end{equation*}
    By the independence the sequence $\{U_j\}_j$, for any $y_1,...,y_j \in \N_0$, we have
    \begin{align*}
        &\P\{L(t) = n, S_{n+1} - t > y_1 , U_{n+2}>y_2,...,U_{n+j}>y_j\}\\
        &= \P\{S_n \leq t < S_{n+1}, S_{n+1} - t > y_1 , U_{n+2}>y_2,...,U_{n+j}>y_j\} \\
        &= \P\{S_n \leq t < S_{n+1}, S_{n+1}-t > y_1\} q^{y_2} \cdot ... \cdot q^{y_j}.
    \end{align*}
    But 
    \begin{align*}
        \P\{S_n \leq t < S_{n+1}, S_{n+1}-t > y_1\} &= \P\{S_n \leq t, U_{n+1} > t-S_n, U_{n+1}> y_1 + t - S_n\}\\
        &= \P\{S_n \leq t, U_{n+1}> y_1 + t - S_n\} \\
        &= \sum_{x = 0}^t \P\{U_{n+1} > y_1 + t - x\} \P\{S_n = x\} \\
        &= q^{y_1} \sum_{x=0}^t \P\{U_{n+1} > t-x\} \P\{S_n = x\} \\
        &= q^{y_1} \P\{S_n \leq t, U_{n+1} > t-S_n\} \\
        &= q^{y_1} \P\{L(t) = n\}.
    \end{align*}
    Hence, 
    \begin{equation*}
        \P\{L(t) = n, S_{n+1} - t > y_1 , U_{n+2}>y_2,...,U_{n+j}>y_j\} = \P\{L(t) = n\}q^{y_1} \cdot ... \cdot q^{y_j}, 
    \end{equation*}
    which, in turn, implies
    \begin{align*}
        \P\{L(t) = n, (U_1^t, ..., U_j^t) \in H\} &= \P\{L(t) = n, S_{n+1} - t > y_1 , U_{n+2}>y_2,...,U_{n+j}>y_j\}\\
        &=\P\{L(t) = n\} \P\{(U_1,...,U_j) \in H\},
    \end{align*}
    being $H = (y_1,\infty)\times ... \times (y_j, \infty)$. Therefore, the law of $(U_1^t,...,U_j^t)$ conditioned on $L(t) = n$ is the same as the law of the vector $(U_1,...,U_j)$. From here, it follows that, for any $n_0,...n_k \in \N_0$
    \begin{align*}
        &\P\{L(t_0) = n_0, L(t_1) - L(t_0) = n_1,..., L(t_k) - L(t_{k-1}) = n_k\}\\
        &= \P\{L(t_0) = n_0, L(t_1) = n_0 + n_1,..., L(t_k) = n_0+...+n_k\} \\
        &= \P\{L(t_0) = n_0, S_{n_0 + ... + n_j} \leq t_j < S_{n_0 + ...+n_{j + 1}}, j=1,...,k\} \\
        &= \P\{L(t_0 )=n_0, U_1^{t_0}+...+U_{n_j}^{t_0} \leq t_j - t_0 < U_1^{t_0}+...+U_{n_{j +1}}^{t_0}, j=1,...,k\} \\
        &= \P\{L(t_0) = n_0, (U_{1}^{t_0},...,U_{n_k + 1}^{t_0}) \in \Tilde{H}\} \\
        &= \P\{L(t_0) = n_0\} \P\{(U_1,...,U_{n_k + 1}) \in \Tilde{H}\} \\
        &= \P\{L(t_0) = n_0\} \P\{L(t_1 - t_0) = n_1,...,L(t_k - t_0) = n_1 +...+n_k\} \\
        &= \P\{L(t_0) = n_0\} \P\{L(t_1 - t_0) =n_1,L(t_2 - t_0) - L(t_1 - t_0) = n_2,...,L(t_k - t_0) - L(t_{k-1} - t_0) = n_k\},
    \end{align*}
    where
    \begin{equation*}
        \Tilde{H} = \{(u_1,...,u_{n_k + 1}) \colon u_1+...+u_{n_1 + ... +n_j} \leq t_j - t_0 < u_1+...+u_{n_1+...+n_j + 1}, j=1,...,k\}.
    \end{equation*}
    Inductively, we can show that
    \begin{align*}
        &\P\{L(t_0) = n_0, L(t_1) - L(t_0) = n_1,..., L(t_k) - L(t_{k-1}) = n_k\} \\
        &= \P\{L(t_0) = n_0\} \prod_{j=1}^k \P\{ L(t_j-t_{j-1}) = n_j\}.
    \end{align*}
    The only thing left to do is to recall that $L(t - s) \sim L(t) - L(s)$ whenever $0 \leq s \leq t$ with $s,t \in \N_0$.
\end{proof}
\section*{Competing Interests}

The authors declare that they have no competing interests.
\printbibliography
\end{document}